\documentclass{article}
\usepackage{amsmath, amsthm, amsfonts, amssymb}
\usepackage{hyperref}
\newtheorem{theorem}{Theorem}[section]
\newtheorem{lemma}[theorem]{Lemma}
\newtheorem{proposition}[theorem]{Proposition}
\newtheorem{corollary}[theorem]{Corollary}
\theoremstyle{definition}
\newtheorem{definition}[theorem]{Definition}
\newtheorem{example}[theorem]{Example}

\theoremstyle{remark}
\newtheorem{remark}[theorem]{Remark}
\numberwithin{equation}{section}

\begin{document}
\title{\bf Bicomplex Weighted Hardy Spaces and Bicomplex $C\textbf{*}$-algebras}
\date{\textbf{Romesh Kumar, Kulbir Singh, Heera Saini, Sanjay Kumar}}
\vspace{0in}
\maketitle

$\textbf{Abstract.}$ In this paper we study the bicomplex version of weighted Hardy spaces. Further, we describe reproducing kernels for the bicomplex weighted Hardy spaces. In particular, we generalize some results which holds for the classical weighted Hardy spaces. We also introduce the notion of bicomplex $C\textbf{*}$-algebra and discuss some of its properties including maximal ideals in $\mathbb {BC}$.\\\\ 
$\textbf{Keywords.}$ Bicomplex modules, hyperbolic norm, bicomplex functional analysis, maximal ideals, $C\textbf{*}$-algebras, weighted Hardy spaces, reproducing kernels. 

\begin{section} {Introduction and Preliminaries}
 \renewcommand{\thefootnote}{\fnsymbol{footnote}}
\footnotetext{ The research of the second author is supported by CSIR-UGC (New-Delhi, India).}
 The theory of Hardy spaces has been a central part of the modern harmonic analysis. Hardy spaces are the spaces of holomorphic functions on different domains in $\mathbb C$ or in $\mathbb C^n$. In this paper we study the bicomplex version of the weighted Hardy spaces with a hyperbolic $(\mathbb D$-valued) norm which is a generalization of real-valued norm. In section 2, we define bicomplex $C\textbf{*}$-algebra and describe its relation with the classical $C\textbf{*}$-algebras. Further, we derive some properties of the bicomplex $C\textbf{*}$-algebras and introduce the notion of bicomplex quotient $C\textbf{*}$-algebras. In section 3, we define composition operators on the bicomplex Hardy space $H^2(\mathbb{U}_{\mathbb{B}\mathbb{C}})$. Further, we estimated the norm of a bounded composition operator induced by a bicomplex disc automorphism on bicomplex Hardy space $H^2(\mathbb{U}_{\mathbb{B}\mathbb{C}})$. Section 4 deals with the bicomplex version of the weighted Hardy spaces. We derive some results including the reproducing kernels for the bicomplex weighted Hardy spaces. For a study of the weighted Hardy spaces, one can refer to \cite{EE}, \cite{shield}, and many other references therein.\\  

Now we summarize some basic properties of bicomplex numbers and bicomplex holomorphic functions.
Let $i$ and $j$ be two imaginary units such that
$ ij=ji,  i^2=j^2=-1$. The set of bicomplex numbers $\mathbb{B}\mathbb{C}$ is then defined as 
\begin{align}\mathbb{B}\mathbb{C}&=\left\{Z=z+jw\;|\;z,w\in \mathbb{C}(i)\right\}
\end{align} where $\mathbb{C}(i)$ is the set of complex numbers with the imaginary unit $i$.
The set $\mathbb{B}\mathbb{C}$ turns out to be a ring with respect to the sum and product  defined by\\
 $$Z+U=(z+jw)+(u+jv)=(z+u)+j(w+v)\;,$$ $$ZU=(z+jw)(u+jv)=(zu-wv)+j(wu+zv)$$ and thus it turns out to be a module over itself. The ring $\mathbb C(i)$ is a subring of $\mathbb{B}\mathbb{C}$, which is in fact a field and thus $\mathbb{B}\mathbb{C}$ can be seen as a vector space over $\mathbb{C}(i)$. For further details on bicomplex analysis, we refer the reader to \cite{YY}, \cite{GG},  \cite{MM}, \cite{M_1M_1}, \cite{ZZ}, \cite{KK},  \cite{RR}, \cite{XX} and references therein.  
 In (1.1), if $z=x$ is real and $w=iy$ is a purely imaginary number with $ij=k$, then we obtain the set of hyperbolic numbers 
 $$\mathbb{D}=\left\{x+ky:k^2=1 \;\text{and}\;x,y\in \mathbb{R}\; \text{with}\;k \notin \mathbb R\right\}.$$
Since bicomplex numbers are defined as the pair of two complex numbers connected through another imaginary unit, there are several natural notions of conjugation.
Let $Z=z+jw\in \mathbb{B}\mathbb{C}$. Then the following three conjugates can be defined in $\mathbb{B}\mathbb{C}$:\\
(i)$Z^{\dagger_1}=\overline{z}+j\overline{w}$\;,\;\; (ii)\;$Z^{\dagger_2}=z-jw$\;,\;\;(iii)\;$Z^{\dagger_3}= \overline{z}-j\overline{w},$\;\;
where $\overline{z}$, $\overline{w}$ denote the usual complex conjugates to $z, w$ in $\mathbb C(i).$  
With each kind of the three conjugations, three possible moduli of a bicomplex number are defined as follows: 
(i) $|Z|^2_j=Z\;.\;Z^{\dagger_1},$
(ii) $|Z|^2_i=Z\;.\;Z^{\dagger_2},$  
(iii) $|Z|^2_k=Z\;.\;Z^{\dagger_3}.$\\ 	 
 A bicomplex number $Z=z+jw$ is said to be invertible if $|Z|_i\neq 0$, that is, $z^2+w^2\neq 0$ and its inverse is given by $Z^{-1}=\frac{Z^{\dagger_2}}{|Z|^2_i}.$ If both $z$ and $w$ are non-zero but the sum $z^2+w^2=0$, then $Z$ is a zero-divisor in $\mathbb {BC}$. We denote the set of all zero divisors in $\mathbb{B}\mathbb{C}$ by $\mathcal {NC}$, that is, $$\mathcal {NC} =\left\{Z\;|\; Z\neq 0,\; z^2+w^2=0 \right\}.$$
The algebra of $\mathbb{B}\mathbb{C}$ is not a division algebra, since one can see that if $e_1=\frac{1+ij}{2}$ and $e_2=\frac{1-ij}{2}$, then $e_1.\;e_2=0$, i.e., $e_1$ and $e_2$ are zero divisors. The numbers $e_1,e_2$ are mutually complementary idempotent components. They make up the so called idempotent basis of bicomplex numbers. Thus, every bicomplex number $Z=z+jw$ can be written in a unique way as:  
\begin{align}Z&=e_1z_1+e_2z_2 \;,
\end{align}  where $z_1=z-iw$ and $z_2=z+iw$ are elements of  $\mathbb{C}(i)$. Formula (1.2) is called the idempotent representation of a bicomplex number $Z$. Further, the two sets $e_1\mathbb {BC}$ and $e_2\mathbb {BC}$ are (principal) ideals in the ring $\mathbb {BC}$ such that $e_1\mathbb {BC} \; \cap \;e_2\mathbb {BC}= \left\{0\right\}$. Thus, we can write \begin{align} \mathbb {BC}= e_1\mathbb {BC}+e_2\mathbb {BC}\;.\end{align} Formula (1.3) is called the idempotent decomposition of $\mathbb {BC}$. In fact, $e_1, e_2$ are hyperbolic numbers. Thus any hyperbolic number $\alpha=x+ky$ can be written as $$\alpha=e_1\alpha_1 +e_2\alpha_2,$$ where $\alpha_1=x+y$ and $\alpha_2=x-y$ are real numbers. 
The set of non-negative hyperbolic numbers is given by (see, e.g., \cite [P. 19] {YY}),$$\mathbb D^+=\left\{\alpha=e_1\alpha_1+e_2\alpha_2 \;|\;\alpha_1, \alpha_2 \geq 0\right\}.$$ For any $\alpha, \beta \in \mathbb D$, we write $\alpha \leq'\beta$ whenever $\beta-\alpha\in \mathbb D^+$ and it defines a partial order on $\mathbb D$. Let $Z=z+jw \in \mathbb {B}\mathbb {C}$, the Euclidean norm of a bicomplex number $Z$ is defined by
$$|Z|=\sqrt{|z|^2+|w|^2}\;.$$
Further, the hyperbolic modulus $|Z|_k$ of a bicomplex number $Z$ is given by $|Z|^2_k=Z\;.\;Z^{\dagger_3}.$
Thus, writing $Z=e_1z_1+e_2z_2$, we have $$|Z|_k=e_1|z_1|+e_2|z_2|$$ and is called the hyperbolic ($\mathbb D$-valued) norm of $Z$. The hyperbolic norm and the Euclidean norm in $\mathbb{B}\mathbb{C}$ has been intensively discussed in \cite [Section 1.3] {YY}.\\
Let $X$ be a $\mathbb{B}\mathbb{C}$-module. Then we can write (see, \cite{GG}, \cite{XX})
\begin{align}X=e_1X_1+e_2X_2\;,\end{align} where $X_1=e_1X$ and $X_2=e_2X$ are complex-linear spaces as well as $\mathbb{B}\mathbb{C}$-modules. Formula (1.4) is called the idempotent decomposition of $X$. Thus, any $x$ in $X$ can be uniquely written as $x=e_1x_1+e_2x_2$ with $x_1\in X_1,\;x_2\in X_2$. Assume that $X_1$, $X_2$ are normed spaces with respective norms $\|.\|_1,\;\|.\|_2$. For any $x \in X$, set 
\begin{align} \|x\|=\sqrt{\frac{\|x_1\|^2_1+\|x_2\|^2_2}{2}}\;\;.\end{align}
 Then $\|.\|$ defines a real-valued norm on $X$ such that for any $\lambda \in \mathbb{B}\mathbb{C}$,\; $x\in X$, it is 
$\|\lambda x\|\leq \sqrt{2}|\lambda|\;\|x\|$. The $\mathbb{B}\mathbb{C}$-module $X$ can also be endowed canonically with the hyperbolic ($\mathbb D$-valued) norm given by the formula  \begin{align} \|x\|_{\mathbb D}=\|e_1x_1+e_2x_2\|_{\mathbb D}=e_1\|x\|_1+e_2\|x_2\|_2 \;,\end{align} such that for any $\lambda\in \mathbb{B}\mathbb{C},\;x\in X$, we have $\|\lambda x\|_{\mathbb D}=|\lambda|_k\;\|x\|_{\mathbb D}$. 
Note that if the $\mathbb{B}\mathbb{C}$-module $X$ has a hyperbolic norm $\|.\|_{\mathbb D}$, then the formula
 \begin{align} |\|x\|_{\mathbb D}|=\|x\| 
\end{align} defines a real-valued norm on $X$.
For more details on hyperbolic ($\mathbb D$-valued) norm see, \cite [Section 4.2] {YY}.\\
 We now define $\mathbb{B}\mathbb{C}$-inner product on $\mathbb{B}\mathbb{C}$-module $X$ which is completely characterized by the inner products on its idempotent components $X_1\; \text{and}\; X_2$. Write $X=e_1X_1+e_2X_2$ and assume $X_1,\;X_2$ are inner product spaces, with inner product  $<.,.>_1$,\; $<.,.>_2$ respectively and corresponding norms $ \|.\|_1\;\text{and}\; \|.\|_2.$ Then the formula
 \begin{align}
 <x,y>_X&=<e_1x_1+e_2x_2,e_1y_1+e_2y_2>_X\notag\\
&= e_1<x_1,y_1>_1+\;\;e_2<x_2,y_2>_2\;,
\end{align}
defines a $\mathbb B\mathbb C$-inner product on $X$. Moreover, the inner product (1.8) introduces a hyperbolic norm on $\mathbb{B}\mathbb{C}$-module $X$ defined as : 
\begin{align}\|x\|_{\mathbb D}=\|e_1x_1+e_2x_2\|_{\mathbb D}=<x,x>_X^\frac{1}{2}\;\;.\end{align}

A $\mathbb B \mathbb C$-inner product in $\mathbb B \mathbb C$ can also be given by the formula  
\begin{align}<Z,W>=Z\;.\;W^{\dagger_3}\;.\end{align}
Note that the inner product (1.10) coincides with the inner-product (1.8) on the $\mathbb{B}\mathbb{C}$-module $\mathbb{B}\mathbb{C}$.
  For more details, see, (\cite  [section 4.3] {YY} and \cite{GG}).
 \begin{theorem}$(\;\text{\cite [Theorem 3.5] {KS}}\;)$ Let $(X, \|.\|)$ be a $\mathbb {BC}$-module. Then $X=e_1X_1+e_2X_2$ is a bicomplex Banach-module if and only if $X_1$, $X_2$ are complex Banach-spaces.
\end{theorem}
Moreover, it is clear that if $(X, \|.\|_\mathbb D)$ is a $\mathbb {BC}$-module then $X$ is a $\mathbb D$-normed bicomplex Banach module if and only if $X_1$, $X_2$ are complex Banach-spaces.     
\begin{definition} \cite  [P. 53] {YY} A $\mathbb{B}\mathbb{C}$-module $X$ with inner product $<.,.>_X$ is said to be a bicomplex Hilbert space if it is complete with respect to the $\mathbb D$-valued norm generated by the bicomplex inner product. Thus, $X=e_1X_1+e_2X_2$ is a bicomplex Hilbert space if and only if $(X_1,<,.,>_1)$ and $(X_2,<,.,>_2)$ are complex Hilbert spaces.
\end{definition}  
 
Let $\Omega$ be a domain in $\mathbb B\mathbb C$. Define $\Omega_1=\left\{z_1\;|\;e_1z_1+e_2z_2\in \Omega\right\}$ and $\Omega_2=\left\{z_2\;|\;e_1z_1+e_2z_2\in \Omega\right\}.$ Then $\Omega_1$ and $\Omega_2$ are domains in $\mathbb C(i)$. For the following Definition 1.3 and the proof of the Theorem 1.4, one can refer to \cite{ZZ} and \cite{KK}.
  \begin{definition} Let $f:\Omega\subseteq \mathbb B\mathbb C\rightarrow \mathbb B\mathbb C$ be a bicomplex function. Then $f$ is said to be a bicomplex holomorphic function in a domain $\Omega$ if it admits a bicomplex derivative at each point of the domain $\Omega$. 
  \end{definition}    
\begin{theorem} Let $\Omega$ be a cartesian domain in $\mathbb B\mathbb C$ determined by $\Omega_1$ and $\Omega_2$. A bicomplex-valued function $f:\Omega\subseteq \mathbb B\mathbb C\rightarrow \mathbb B\mathbb C$ is a bicomplex holomorphic function if and only if there exists two complex-valued holomorphic functions $f_1$ and $f_2$ on $\Omega_1$ and $\Omega_2$ respectively such that $$ f(Z)=f(e_1z_1+e_2z_2)=e_1f_1(z_1)+e_2f_2(z_2).$$  
\end{theorem}
 For further details on bicomplex holomorphic functions, (see, \cite{YY}, \cite{H_6H_6}, \cite{MM}, \cite{ZZ}, \cite{KK}, \cite{Z_1Z_1}).\\ G. B. Price book \cite{KK} contains an extensive survey of the various fundamental properties of bicomplex numbers and bicomplex function theory. Also, D. Alpay et al. \cite{YY} have given a interesting detail survey of bicomplex functional analysis with the introduction of different norms and inner-products on the bicomplex modules. Recently, in \cite{Hahn}, the Hahn-Banach theorems and their consequences for $\mathbb D$-modules and $\mathbb {BC}$-modules are proved. Further, some properties of bicomplex linear operators over bicomplex Hilbert spaces are discussed in \cite{KS}.     
  \end{section}
\begin{section} {Bicomplex $C\textbf{*}$-algebras}
In this section we define bicomplex $C\textbf{*}$-algebra with $\mathbb D$-valued as well as with real-valued norm. Further, we discuss some of its properties. 
\begin{definition} A bicomplex algebra $\mathcal A$ is a module over $\mathbb {BC}$  with a multiplication defined on it which satisfy the following properties\;:    
\begin{enumerate}
\item[(i)] $x(y+z)=xy+xz$ 
\item[(ii)] $(x+y)z=xy+yz$
\item[(iii)] $x(yz)=(xy)z$ 
\item[(iv)] $\lambda(xy)=(\lambda x)y=x(\lambda y)$
   
\end{enumerate}
for all $x,\;y,\;z \in \mathcal A$ and all scalars  $\lambda\in \mathbb {BC}$.
\end{definition}

 \begin{definition} A bicomplex algebra $\mathcal A$ with a $\mathbb D$-valued norm $\|.\|_\mathbb D$ relative to which $\mathcal A$ is a Banach module and such that for every  $x,y$ in $\mathcal A$,   
$$ \|xy\|_\mathbb D \leq' \|x\|_\mathbb D\|y\|_\mathbb D$$ 
is called a $\mathbb D$-normed bicomplex Banach algebra.
\end{definition}
\begin{remark} A bicomplex algebra $\mathcal A$ with a real-valued norm $\|.\|$ relative to which $\mathcal A$ is a Banach module and for every $x,y$ in $\mathcal A$,   $$\|xy\|\leq \sqrt{2}\;\|x\|\|y\|$$ is called a bicomplex Banach algebra.
\end{remark}

\begin{definition} A mapping $x\longrightarrow x\textbf{*}$ of a bicomplex algebra $\mathcal A$ into $\mathcal A$ is called an involution on $\mathcal A$ if the following properties hold for all $x,y\in \mathcal A$ and $\lambda \in \mathbb B \mathbb{C}$\;:
\begin{enumerate}

\item[(i)] $(x^\textbf{*})^\textbf{*}=x$
\item[(ii)] $(xy)^\textbf{*}=y^\textbf{*} x^\textbf{*}$
\item[(iii)] $(\lambda x+y)^\textbf{*}=\lambda^{\dagger_3} x^\textbf{*}+y^\textbf{*}$.
\end{enumerate}
Involution function can also be defined with $\dagger_1$ and $\dagger_2$-conjugations. For more details one can refer to (\cite [P. 40] {YY}). Note that in this paper we used involution function with $\dagger_3$-conjugation. 
\end{definition}
\begin{remark} It is easy to show that a bicomplex algebra $\mathcal A$ can be written as $\mathcal A=e_1\mathcal A_1+e_2\mathcal A_2$, where $\mathcal A_1,\; \mathcal A_2$ are the algebras over $\mathbb C(i)$. Thus for any $x\in \mathcal A$, we can write $x=e_1x_1+e_2x_2$, where $x_1\in \mathcal A_1$ and $x_2\in \mathcal A_2$. Further,
 
$x\textbf *=(e_1x_1+e_2x_2)\textbf *
 =e_1^{\dagger_3}x_1\textbf *+e_2^{\dagger_3}x_2\textbf *=e_1x_1\textbf *+e_2x_2\textbf *.$
  
 Hence $_\textbf{*}$ is an involution on a bicomplex algebra $\mathcal A$ if and only if $_\textbf{*}$ is an involution on $\mathcal A_1$ and $\mathcal A_2$.
 \end{remark} 
 \begin{definition} A $\mathbb D$-normed bicomplex Banach algebra $\mathcal A$ with an involution such that for every $x$ in $\mathcal A$
$$\|x\textbf{*}x\|_\mathbb D=\|x\|^2_\mathbb D$$ is called a $\mathbb D$-normed bicomplex $C\textbf{*}$-algebra.
\end{definition}
\begin{remark} 
 A bicomplex Banach-algebra with an involution such that for every $x$ in $\mathcal A$,
$$\|x\|^2\leq\|x\textbf{*}x\|\leq \sqrt{2}\|x\|^2$$ is called a bicomplex $C\textbf{*}$-algebra. 
\end{remark} 
\begin{theorem} Let $\mathcal A$ be a $\mathbb D$-normed bicomplex algebra and $\mathcal A=e_1\mathcal A_1+e_2\mathcal A_2$ be its idempotent decomposition. Then $\mathcal A$ is a $\mathbb D$-normed bicomplex $C\textbf{*}$-algebra if and only if $\mathcal A_1,\;\mathcal A_2$ are complex $C\textbf{*}$-algebras.
\end{theorem}  
\begin{proof}
By using Theorem 1.1, any $\mathbb D$-normed bicomplex Banach module $\mathcal A$ can be written as $\mathcal A=e_1\mathcal A_1+e_2\mathcal A_2$, where $\mathcal A_1,\; \mathcal A_2$ are complex Banach spaces. Let $x, y\in \mathcal A$ with $x=e_1x_1+e_2x_2,\; y=e_1y_1+e_2y_2$, where $x_1, y_1 \in \mathcal A_1$ and $x_2, y_2\in \mathcal A_2$. Then   
\begin{align*}   &~~~~~\|xy\|_\mathbb D \leq' \|x\|_\mathbb D\|y\|_\mathbb D\\
&\Leftrightarrow \|e_1x_1y_1+e_2x_2y_2\|_\mathbb D\leq'\|e_1x_1+e_2x_2\|_\mathbb D\|e_1y_1+e_2y_2 \|_\mathbb D\\
&\Leftrightarrow e_1\|x_1y_1\|_1 +e_2\|x_2y_2\|_2  \leq' e_1\|x_1\|_1\|y_1\|_1 +e_2\|x_2\|_2\|y_2\|_2\\
&\Leftrightarrow \|x_1y_1\|_1\leq\|x_1\|_1\|y_1\|_1\; \text{and}\; \|x_2y_2\|_2\leq\|x_2\|_2\|y_2\|_2. 
\end{align*} 
Hence $\mathcal A$ is a $\mathbb D$-normed bicomplex Banach algebra if and only if $\mathcal A_1,\;\mathcal A_2$ are complex Banach algebras. Further, for any $x\in \mathcal A,$
\begin{align*}   &~~~~~\|x\textbf{*}x\|_\mathbb D=\|x\|^2_\mathbb D\\
&\Leftrightarrow \|e_1x_1\textbf * x_1 +e_2x_2\textbf * x_2 \|_\mathbb D=\|e_1x_1+e_2x_2\|^2_\mathbb D\\
&\Leftrightarrow e_1\|x_1\textbf * x_1\| +e_2\|x_2\textbf * x_2\|=e_1\|x_1\|^2+e_2\|x_2\|^2\\
&\Leftrightarrow \|x_1\textbf * x_1\|=\|x_1\|^2\;\text{and}\;\|x_2\textbf * x_2\|=\|x_2\|^2. 
\end{align*}
Hence a $\mathbb D$-normed bicomplex $C\textbf{*}$-algebra can be written as $\mathcal A=e_1\mathcal A_1+e_2\mathcal A_2$, where $\mathcal A_1,\;\mathcal A_2$ are complex $C\textbf{*}$-algebras.  
\end{proof}   
\begin{proposition} Let $\mathcal A$ be a $\mathbb D$-normed bicomplex $C\textbf{*}$-algebra and $x\in \mathcal A$. Then $$\|x\textbf *\|_\mathbb D=\|x\|_\mathbb D.$$
\end{proposition}
 Now we give some $\textbf{examples}$ of bicomplex $C^\textbf{*}$- algebras.\\\\
(i) Consider the set $\mathbb {BC}$  with the usual multiplication of bicomplex numbers. Then $\mathbb {BC}$ is a $\mathbb {BC}$-module. Let $|.|_k$ denotes the $\mathbb D$-norm over $\mathbb {BC}$. For each $Z\in\mathbb {BC}$, define an involution on $Z$ as $Z\textbf{*}=Z^{\dagger_3}$. Then $\mathbb {BC}$ is a $\mathbb D$-normed bicomplex $C^\textbf{*}$- algebra. Further, $\mathbb {BC}^n$ over $\mathbb {BC}$ also forms a $\mathbb D$-normed bicomplex $C^\textbf{*}$- algebra.\\\\  
(ii) Let $X$ be a bicomplex Hilbert space. Then the space $\mathcal B(X)$ of all $\mathbb D$-bounded linear operators from $X$ into itself with a $\mathbb D$-valued norm defined as $\|T\|_{\mathbb D}=\text{sup}\left\{ \|T(x)\|_{\mathbb D}\;|\;x\in X,\; \|x\|_{\mathbb D}\leq' 1\right\}$ is a $\mathbb D$-normed bicomplex $C^\textbf{*}$- algebra, where for each $T$ in $\mathcal B(X)$, $T^\textbf{*}$ is the adjoint of $T$.\\\\
(iii)  Let $\Omega$ be a compact topological space. Let $C(\Omega,\mathbb {BC})$ denotes the space of all bicomplex-valued continuous functions on $\Omega$. Then  $C(\Omega,\mathbb {BC})$ with a $\mathbb D$-valued norm defined as $\|f\|_{\mathbb D}=\text{sup} \left\{|f(x)|_k\;|\;x\in \Omega\right\}$ is a $\mathbb D$-normed bicomplex $C^\textbf{*}$- algebra, where $f^\textbf{*}(x)=f^{\dagger_3}(x)$, $\forall\; f \in C(\Omega,\mathbb {BC})$.\\\\ 
(iv) Let $\mathbb{U}_{\mathbb{B}\mathbb{C}}=\left\{Z=z+jw=e_1z_1+e_2z_2\;|\;(z_1,z_2)\in \mathbb{U}^2\right\}$
 be the unit disc in $\mathbb{B}\mathbb{C}$,
where $\mathbb{U}$ is the unit disk in the complex plane $\mathbb C(i)$. Let $H^\infty(\mathbb{U}_{\mathbb{B}\mathbb{C}})$ denotes the space of all $\mathbb D$-bounded bicomplex holomorphic functions with a $\mathbb D$-valued norm defined as $$\|f\|_{\mathbb D,H^\infty(\mathbb{U}_{\mathbb{B}\mathbb{C}})}=\text{sup}\left\{|f(x)|_k\;|\;x\in \mathbb{U}_{\mathbb{B}\mathbb{C}} \right\}.$$ Further, for each $f\in H^\infty(\mathbb{U}_{\mathbb{B}\mathbb{C}})$ and $Z\in \mathbb{U}_{\mathbb{B}\mathbb{C}}$, define an involution on $f$ as $f^\textbf{*}(Z)=f^{\dagger_3}(Z)$. Then $H^\infty(\mathbb{U}_{\mathbb{B}\mathbb{C}})$ is a $\mathbb D$-normed bicomplex $C^\textbf{*}$- algebra.

\begin{definition} Let $\mathcal A$ be a $\mathbb D$-normed bicomplex Banach algebra. Then a $\mathbb {BC}$-homomorphism $f:\mathcal A\rightarrow \mathbb {BC}$ is a $\mathbb {BC}$-linear map such that $f(xy)=f(x)f(y)$ \;$\forall\; x,y \in \mathcal A$.
\end{definition}   
\begin{definition} Let $\mathcal A$ and $\mathcal B$ be the two $\mathbb D$-normed bicomplex $C\textbf{*}$-algebras and $f:\mathcal A\rightarrow \mathcal B$ be a $\mathbb {BC}$-algebra homomorphism. Then $f$ is a $\mathbb {BC}$ $_\textbf *$-homomorphism if $$f(x\textbf*)=f(x)\textbf*,\;\; \forall\; x \in \mathcal A.$$
\end{definition}
\begin{remark} A $\mathbb {BC}$-linear map $f$ can be written as $f=e_1f_1+e_2f_2$, where $f_1, f_2$ are the $\mathbb C(i)$-linear maps over $\mathcal A_1, \mathcal A_2$ respectively (see, \cite{YY}, \cite{HH}). Now for any $x=e_1x_1+e_2x_2,\; y=e_1y_1+e_2y_2 \in \mathcal A$, 
  $$f(xy)=f(x)f(y) 
\Leftrightarrow f_1(x_1y_1)=f_1(x_1)f_1(y_1)\;\text{and}\; f_2(x_2y_2)=f_2(x_2)f_2(y_2).$$ Further, we observe that	    
\begin{align*}  \;\; &f(x\textbf*)=f(x)\textbf*\\ 
\Leftrightarrow & (f_1e_1+f_2e_2)(x_1\textbf*e_1+x_2\textbf*e_2)=(f_1(x_1)e_1+f_2(x_2)e_2)\textbf*\\
\Leftrightarrow & f_1(x_1\textbf*)= f_1(x_1)\textbf*\;\text{and}\; f_2(x_2\textbf*)= f_2(x_2)\textbf*.
\end{align*}
Thus $f$ is a $\mathbb {BC}$ $_\textbf *$-homomorphism over $\mathcal A$ if and only if $f_1$, $f_2$ are $_\textbf *$-homomorphisms over $\mathcal A_1$, $\mathcal A_2$ respectively. 
\end{remark}
\begin{definition} Let $\mathcal A$ be a $\mathbb D$-normed bicomplex $C\textbf{*}$-algebra and $x\in \mathcal A$. Then
\begin{enumerate}
\item[(i)]$x$ is said to be self-adjoint (hermitian) if $x=x\textbf *$.
\item[(ii)] $x$ is said to be normal if $x\textbf *x=xx\textbf *$.
\item[(iii)] $x$ is said to be unitary if $x\textbf *x=xx\textbf *=1$.
\end{enumerate}
\end{definition}
\begin{proposition} Let $\mathcal A$ be a $\mathbb D$-normed $C\textbf{*}$-algebra and let $x \in \mathcal A$.\\
$(i)$ If $x$ is invertible, then $x\textbf{*}$ is invertible and $(x\textbf{*})^{-1}$ = $(x^{-1})\textbf{*}$.\\
$(ii)$ $x=u+iv$, where $u$ and $v$ are hermitian elements of $\mathcal A$.\\  
$(iii)$ If $x$ is unitary element of $\mathcal A$, then $\|x\|_{\mathbb D}=1$.
\end{proposition}
\begin{proof} Properties (i) and (iii) holds trivially. For (ii), let $x\in \mathcal A$. Then we can write $x=e_1x_1+e_2x_2$, where $x_1, x_2\in \mathcal A_1, \mathcal A_2$ respectively.
 Now $x_l\in \mathcal A_l$ and $\mathcal A_l$ is a complex $C\textbf{*}$-algebra, we can write $x_l=u_l+iv_l$, where $u_l, v_l$ are the hermitian elements of $A_l$ for $l=1,2$. Thus,  
 \begin{align*} x&=e_1x_1+e_2x_2
 =e_1(u_1+iv_1)+e_2(u_2+iv_2)\\
 &=(e_1u_1+e_2u_2)+i(e_1v_1+e_2v_2)=u+iv,
 \end{align*}
 where $u=e_1u_1+e_2u_2$ and $v=e_1v_1+e_2v_2$ are the hermitian elements of $\mathcal A$.
 \end{proof}
 \begin{proposition} Let $\mathcal A$ be a $\mathbb D$-normed bicomplex $C\textbf{*}$-algebra and $f:\mathcal A\rightarrow \mathbb {BC}$ be a non-zero $\mathbb {BC}$-homomorphism. Then\\
$(i)$ $f(x)\in \mathbb D$, whenever $x=x\textbf{*}$.\\
$(ii)$ $f(x\textbf {*})= (f(x))^{\dagger_3},\; \forall\; x \in \mathcal A$.\\
$(iii)$ $f(x\textbf{*}x) \in \mathbb D^+,\; \forall\; x \in \mathcal A$.\\ 
$(iv)$ if $1\in \mathcal A$ and $x$ is unitary element of $A$, then $|f(x)|_k=1$. 
\end{proposition}
\begin{proof} Let $f=e_1f_1+e_2f_2$ be a non-zero $\mathbb {BC}$-homomorphism on $\mathcal A$, where $f_1, f_2$ are homomorphisms on $\mathcal A_1,\; \mathcal A_2$ respectively. If either one of $f_1$ or $f_2$ are zero homomorphism, then the proof trivially holds by using the classical results. Suppose that $f_1, f_2$ both are non-zero homomorphisms.\\
$(i)$ Let $x\in \mathcal A$ such that $x=x\textbf{*}$. This implies that $x_1=x_1\textbf{*}$ and $x_2=x_2\textbf{*}$. Further, for a non-zero homomorphism $f_l: \mathcal A_l\rightarrow \mathbb C(i)$, we have $f_l(x_l)\in \mathbb R$, $l=1,2$. Thus, $f(x)= e_1f_1(x_1)+e_2f_2(x_2)\in \mathbb D$.\\  
$(ii)$ For each $x\in \mathcal A$, we can write $x\textbf{*}=e_1x_1\textbf{*}+e_2x_2\textbf{*}$. Further, for a non-zero homomorphism $f_l: \mathcal A_l\rightarrow \mathbb C(i)$, we have $f_1(x_1\textbf{*})=\overline {f_1(x_1)},\;\forall \;x_1 \in \mathcal A_1$ and $f_2(x_2\textbf{*})=\overline {f_2(x_2)},\;\forall \;x_2 \in \mathcal A_2$. Hence, 
\begin{align*} f(x\textbf{*})&= e_1f_1(x_1\textbf{*})+e_2f_2(x_2\textbf{*})\\
&=e_1\overline {f_1(x_1)}+e_2\overline {f_2(x_2)} =(f(x))^{\dagger_3},\; \forall\; x \in \mathcal A.
\end{align*}
$(iii)$ Since $f(x\textbf{*}x) =f(x\textbf{*})f(x)= (f(x))^{\dagger_3}f(x)=|f(x)|^2_k \in \mathbb D^+,\;\forall \;x\in \mathcal A$.
$(iv)$  If $x$ is unitary, then by using (iii), $|f(x)|^2_k= f(x\textbf{*}x)=f(1)=1$. Hence $|f(x)|_k=1$. 
 \end{proof}
\begin{definition} Let $\mathcal A$ be a bicomplex algebra and $\mathcal I$ be a subalgebra of $\mathcal A$. Then $\mathcal I$ is said to be an ideal (two-sided) of $\mathcal A$ if $xy\in \mathcal I$ and $yx\in \mathcal I$ whenever $x\in \mathcal A$ and $y\in \mathcal I$.
 We can also talk about the left as well as the right ideal of $\mathcal A$. Let $x\in \mathcal A$ and $y\in \mathcal I$, then the left ideal $\mathcal I$ of $\mathcal A$ is a subalgebra of $\mathcal A$ such that $xy\in \mathcal I$ and the right ideal $\mathcal I$ of $\mathcal A$ is a subalgebra of $\mathcal A$ such that $yx\in \mathcal I$. An ideal $\mathcal I\neq \mathcal A$ is called a proper ideal of $\mathcal A$.\end{definition}
\begin{remark}\label{ideal} (\cite [Definition 3.1] {KS}) Let $\mathcal A$ be a $\mathbb D$-normed bicomplex Banach algebra and $\mathcal I$ be an ideal of $\mathcal A$. Then $\mathcal A/\mathcal I$ is a bicomplex quotient module. Moreover, we can write bicomplex quotient module as $\mathcal A/\mathcal I=e_1\mathcal A_1/\mathcal I_1+e_2\mathcal A_2/\mathcal I_2$, where $\mathcal A_1/\mathcal I_1, \mathcal A_2/\mathcal I_2$ are the quotient spaces over $\mathbb C(i)$. Further, it is easy to show that if $\mathcal I$ is closed then $\mathcal A/\mathcal I$ is a $\mathbb D$-normed bicomplex Banach module with the $\mathbb D$-valued quotient norm defined as $\|x+I\|_\mathbb D=\text{inf}\left\{\|x-y\|_\mathbb D : y\in I\right\}$. 
\end{remark}
\begin{definition}\label{qtalg} Let $\mathcal A$ be a $\mathbb{B}\mathbb{C}$-algebra and $\mathcal I$ be an ideal of $\mathcal A$. Then we can write $\mathcal \mathcal \mathcal A=e_1\mathcal A_1+e_2\mathcal A_2$ and $\mathcal I=e_1\mathcal I_1+e_2\mathcal I_2$, where  $\mathcal I_1\;\text{and}\;\mathcal I_2$ are ideals of $\mathcal A_1\;\text{and}\;\mathcal A_2$ respectively. Moreover, $\mathcal A_l/\mathcal I_l$ for $l=1,2$ is the quotient algebra over $\mathbb C(i)$. 
Consider the set $\mathcal A/\mathcal I=\left\{x+\mathcal I:x\in \mathcal I\right\}$. Thus, with $x=e_1x_1+e_2x_2$, we can write $x+\mathcal I=e_1(x_1+\mathcal I_1)+e_2(x_2+\mathcal I_2),$ where $x_1+\mathcal I_1\in \mathcal A_1/\mathcal I_1$ and $x_2+\mathcal I_2\in \mathcal A_2/\mathcal I_2$. For any $x,\;y\in \mathcal A$, we define the multiplicative structure on $\mathcal A/\mathcal I$ by using the multiplication of the quotient algebra $\mathcal A_l/\mathcal I_l$ for $l=1,2$ as: 
\begin{align*} (x+\mathcal I)(y+\mathcal I)&=(e_1(x_1+\mathcal I_1)+e_2(x_2+\mathcal I_2))(e_1(y_1+\mathcal I_1)+e_2(y_2+\mathcal I_2))\\
&=e_1(x_1+\mathcal I_1)(y_1+\mathcal I_1)+e_2(x_2+\mathcal I_2)(y_2+\mathcal I_2)\\
&=e_1(x_1y_1+\mathcal I_1)+e_2(x_2y_2+\mathcal I_2)\\
&=xy+\mathcal I.
\end{align*}
It is easy to check that the multication defined above has the desired associative and distributive properties. Hence $\mathcal A/\mathcal I$ is a bicomplex quotient algebra.
Thus we can write the bicomplex quotient algebra as
$$ \mathcal A/\mathcal I = e_1\mathcal A_1/\mathcal I_1+e_2\mathcal A_2/\mathcal I_2,$$ where $\mathcal A_1/\mathcal I_1$ and $\mathcal A_2/\mathcal I_2$ are the quotient algebras over $\mathbb C(i)$.  
\end{definition}

\begin{theorem}\label{C*alg} Let $\mathcal A$ be a $\mathbb D$-normed bicomplex Banach algebra and $\mathcal I$ be a proper closed ideal of $\mathcal A$. Then $\mathcal A/\mathcal I$ is a $\mathbb D$-normed bicomplex Banach algebra.
\end{theorem}
\begin{proof} Since we have seen that $\mathcal A/\mathcal I$ forms a bicomplex quotient algebra. Also from the Remark \ref{ideal}, $\mathcal A/\mathcal I$ is a $\mathbb D$-normed bicomplex quotient Banach module. Further, $\mathcal I_1$, $\mathcal I_2$ are the closed ideals of $\mathcal A_1$, $\mathcal A_2$ respectively implies that $\mathcal A_1/\mathcal I_1$ and $\mathcal A_2/\mathcal I_2$ are the quotient Banach algebras over $\mathbb C(i)$. To show that $\mathcal A/\mathcal I$ is a $\mathbb D$-normed bicomplex Banach algebra, we need to verify that for any $x, y\in \mathcal A$, $$\|(x+\mathcal I)(y+\mathcal I)\|_\mathbb D\leq^{'} \|x+\mathcal I\|_\mathbb D\|y+\mathcal I\|_\mathbb D.$$ Let $x, y\in \mathcal A$. Then
\begin{align*} \|(x+\mathcal I)(y+\mathcal I)\|_\mathbb D&=\|(e_1(x_1+\mathcal I_1)+e_2(x_2+\mathcal I_2))(e_1(y_1+\mathcal I_1)+e_2(y_2+\mathcal I_2))\|_\mathbb D\\
&=\|e_1(x_1+\mathcal I_1)(y_1+\mathcal I_1)+e_2(x_2+\mathcal I_2)(y_2+\mathcal I_2) \|_\mathbb D\\
&=e_1\|(x_1+\mathcal I_1)(y_1+\mathcal I_1) \|_1+e_2\|(x_2+\mathcal I_2)(y_2+\mathcal I_2)\|_2\\      
&\leq^{'} e_1(\|x_1+\mathcal I_1\|_1\|y_1+\mathcal I_1\|_1)+e_2(\|x_2+\mathcal I_2\|_2\|y_2+\mathcal I_2\|_2)\\
&=(e_1\|x_1+\mathcal I_1\|_1+e_2\|x_2+\mathcal I_2\|_2)(e_1\|y_1+\mathcal I_1\|_1+e_2\|y_2+\mathcal I_2\|_2)\\
&=\|x+\mathcal I\|_\mathbb D\|y+\mathcal I\|_\mathbb D.
\end{align*}
Hence $\mathcal A/\mathcal I$ is a $\mathbb D$-normed bicomplex quotient Banach algebra.
 \end{proof}
\begin{lemma} \label{c*} Let $\mathcal A$ be a $\mathbb D$-normed bicomplex $C\textbf{*}$-algebra and $\mathcal I$ be a closed ideal in $\mathcal A$. Then $\mathcal I$ is a self-adjoint ideal in $\mathcal A$.
 \end{lemma}
 \begin{theorem} Let $\mathcal A$ be a $\mathbb D$-normed bicomplex $C\textbf{*}$-algebra and $\mathcal I$ be a proper closed ideal of $\mathcal A$. Then $\mathcal A/\mathcal I$ is a $\mathbb D$-normed bicomplex quotient $C\textbf{*}$-algebra.
\end{theorem} 
\begin{proof}We can conclude that $\mathcal A/\mathcal I$ is a $\mathbb D$-normed bicomplex quotient Banach algebra. Further, by using Lemma \ref{c*}, $\mathcal I$ is a self-adjoint ideal in $\mathcal A$. Define an involution on $\mathcal A/\mathcal I$ as $(x+\mathcal I)\textbf{*}=x\textbf{*}+\mathcal I$, for each $x+\mathcal I\in \mathcal A/\mathcal I$. Thus for $x=e_1x_1+e_2x_2$, $(x_l+\mathcal I_l)\textbf{*}=x_l\textbf{*}+\mathcal I_l$  defines an involution on $\mathcal A_l/\mathcal I_l$ for $l=1,2$. Hence $\mathcal A_1/\mathcal I_1$ and $\mathcal A_2/\mathcal I_2$ are complex $C\textbf{*}$-algebras. To show that $\mathcal A/\mathcal I$ is a $\mathbb D$-normed bicomplex $C\textbf{*}$-algebra, we must check that the $C\textbf{*}$-identity holds i.e., for all $x\in \mathcal A$, $$\|x+\mathcal I\|^2_\mathbb D=\|(x\textbf{*}+\mathcal I)(x+\mathcal I)\|_\mathbb D.$$     Let $x\in \mathcal A$. Then 
\begin{align*} \|(x\textbf{*}+\mathcal I)(x+\mathcal I)\|_\mathbb D&=\|x\textbf{*}x+\mathcal I\|_\mathbb D\\
&=\|(e_1x_1\textbf{*}+e_2x_2\textbf{*})(e_1x_1+e_2x_2)+(e_1\mathcal I_1+e_2\mathcal I_2)\|_\mathbb D\\
&=\|e_1(x_1\textbf{*}x_1+\mathcal I_1)+e_2(x_2\textbf{*}x_2+\mathcal I_2)\|_\mathbb D\\
&=e_1\|x_1\textbf{*}x_1+\mathcal I_1\|_1+e_2\|x_2\textbf{*}x_2+\mathcal I_2\|_2\\
&=e_1\|x_1+\mathcal I_1\|^2_1+e_2\|x_2+\mathcal I_2\|^2_2\\
&=\|x+\mathcal I\|^2_\mathbb D.
\end{align*}
 Hence the $C\textbf{*}$-identity holds.
 \end{proof}
\begin{example} Since $\mathbb {BC}$ is a $\mathbb D$-normed bicomplex $C^\textbf{*}$- algebra. Let $\mathcal I_1= e_1\mathbb {BC}$. Then $\mathcal I_1$ is a proper closed ideal in $\mathbb {BC}$. Hence $\mathbb {BC}/\mathcal I_1$ is a bicomplex quotient module. Further, $$\mathcal I_1\textbf{*}= (e_1\mathbb {BC})\textbf{*}=(e_1\mathbb {BC})^{\dagger_3}=e_1\mathbb {BC}=\mathcal I_1.$$ Thus $\mathcal I_1$ is a self-adjoint ideal in $\mathbb {BC}$. For any $Z\in \mathbb {BC}$, we define a quotient norm on $\mathbb {BC}/\mathcal I_1$ as:
\begin{align*} |Z+\mathcal I_1|_k&=|e_1z_1+e_2z_2+\mathcal I_1|_k\\
&=|e_2z_2+\mathcal I_1|_k=\text{inf}\left\{|e_2z_2+e_1x_1|_k\;:\;x_1\in \mathcal I_1\right\}\\
&=\text{inf}\left\{ e_2|z_2| +e_1|x_1|\;:\;x_1\in \mathcal I_1\right\}\\
&= e_2\;\text{inf}\;|z_2|+ e_1\;\text{inf}\left\{|x_1|\;:\;x_1\in \mathcal I_1\right\}= e_2|z_2|. 
\end{align*}
Moreover, the $C\textbf{*}$-identity easily follows with the quotient norm defined over $\mathbb {BC}/\mathcal I_1$. Hence $\mathbb {BC}/\mathcal I_1$ is a $\mathbb D$-normed bicomplex quotient  $C^\textbf{*}$- algebra.\\ Similarly for closed ideal $\mathcal I_2= e_2\mathbb {BC}$, $\mathbb {BC}/\mathcal I_2$ also forms a $\mathbb D$-normed bicomplex quotient  $C^\textbf{*}$- algebra.  
\end{example}
Now we discuss the maximal ideals in the ring $\mathbb {BC}$.
Cosider the bicomplex ring $\mathbb {BC}$. There are obviously only four ring homomorphisms $f:\mathbb {BC}\rightarrow \mathbb {BC}$ given as follows:\\
 (i) $f(Z)= Z$, $\forall \;Z\in \mathbb {BC}$, i.e., the identity homomorphism. In this case $Ker f=\left\{0\right\}$, that is, the zero ideal which is an improper ideal in $\mathbb {BC}$. Moreover, if $f(Z)=Z^{\dagger_1}$ or $Z^{\dagger_2}$ or $Z^{\dagger_3}$ then $f$ is a ring homomorphism with $Ker f=\left\{0\right\}$.\\\\
 (ii) $f(Z)=0$, $\forall \;Z\in \mathbb {BC}$, the zero homomorphism. Here $Ker f= \mathbb {BC}$ which is also an improper ideal in $\mathbb {BC}$.\\\\
 (iii) Define $f(Z)=e_1Z$ ( = $e_2Z$), $\forall \;Z\in \mathbb {BC}$. Then $f$ is a ring homomorphism with $Ker f= e_2\mathbb {BC}$ ( = $e_1\mathbb {BC}$ respectively) which is a proper ideal in $\mathbb {BC}$.\\\\
Hence in the ring  $\mathbb {BC}$, there exists two improper ideals and two proper ideals. Next we will show that the two proper ideals in $\mathbb {BC}$ are the maximal ideals in $\mathbb {BC}$.
 
Let $I_1=e_1\mathbb {BC}$ and $I_2=e_2\mathbb {BC}$. Then $I_1,I_2$ are proper ideals in $\mathbb {BC}$. Further, 
 \begin{align*}  \mathbb {BC}/I_1&=\left\{ I_1+Z\;|\;Z=e_1z_1+e_2z_2\in \mathbb {BC}\right\}\\
 &=\left\{ I_1+e_2z_2\;|\;Z=e_1z_1+e_2z_2\in \mathbb {BC}\right\}.
 \end{align*}
 For any $Z=e_1z_1+e_2z_2, W=e_1w_1+e_2w_2\in \mathbb {BC}$, we define addition and multiplication on $\mathbb {BC}/I_1$ as
 $$(I_1+Z) + (I_1+W)= I_1+(Z+W)= I_1+e_2(z_2+w_2),$$
 $$(I_1+Z)(I_1+W)= I_1+(ZW)= I_1+e_2(z_2w_2).$$
Then it is easy to show that $\mathbb {BC}/I_1$ is a commutative ring with unity $I_1+e_2$. Further, for any $Z=e_1z_1+e_2z_2\in \mathbb {BC}$, let $I_1+e_2z_2 \in \mathbb {BC}/I_1$. Since $z_2\in \mathbb C(i)$ and $\mathbb C(i)$ is a field, there exist some $w\in \mathbb C(i)$ such that $z_2w=wz_2=1$. Thus, $I_1+e_2w \in \mathbb {BC}/I_1$ such that $(I_1+e_2w) (I_1+e_2z_2)=(I_1+e_2z_2)(I_1+e_2w)=I_1+e_2$. Hence $\mathbb {BC}/I_1$ is a field.\\  
Thus using a well known result from algebra,\\ ``In a commutative ring R with unity, an ideal M is maximal in R iff R/M is a field", we find that $I_1$ is a maximal ideal in the ring $\mathbb {BC}$. Similarly, $I_2$ is a maximal ideal in the ring $\mathbb {BC}$.\\\\
   \textbf {Aliter} : $I_1, I_2$ are maximal ideals in  $\mathbb {BC}$ can also be proved by using the Fundamental Theorem of ring Homomorphism.\\ 
 Define a function $f:\mathbb {BC}\rightarrow \mathbb C(i)$ such that $f(Z)=f(z+jw)=z+iw$.\\ 
 Then $f$ is a ring homomorphism as $$f(Z+U)=f(Z)+f(U),$$
 $$f(ZU)=f(Z)f(U).$$ Moreover, $f$ is onto. Thus by Fundamental Theorem of homomorphism 
 $$ \mathbb {BC}/Ker f \cong \mathbb C(i).$$
 Now for any $Z\in \mathbb {BC}$,\begin{align*} Z=z+jw\in Ker f&\Leftrightarrow f(Z)=f(z+jw)=0\\
 &\Leftrightarrow z+iw=0\\
 &\Leftrightarrow Z=z+jiz\\
 &\Leftrightarrow Z\in I_1.
\end{align*}
Thus $Ker f= I_1$. Further, $\mathbb {BC}/Ker f \cong \mathbb C(i)$, but $\mathbb C(i)$ being a field, $\mathbb {BC}/Ker f$ will be a field. Hence $I_1$ is maximal ideal in $\mathbb {BC}$.\\ Similarly define a ring homomorphism $f:\mathbb {BC}\rightarrow \mathbb C(i)$ such that $$f(Z)=f(z+jw)=z-iw.$$ Then $I_2$ is also maximal in $\mathbb {BC}$.\\
  
\begin{remark} In a unital commutative bicomplex Banach algebra $\mathcal A$, for every non-zero multiplicative $\mathbb {BC}$-linear functional $f:\mathcal A \rightarrow \mathbb {BC}$, $Ker f$ need not be a maximal ideal of $\mathcal A$.
\end{remark} 
 For example,  consider a unital commutative bicomplex Banach algebra $\mathbb {BC}$ and define $f:\mathbb {BC} \rightarrow \mathbb {BC}$ by $f(Z)=Z$. Then $ker f=\left\{0\right\}$. But $\left\{0\right\}$ is not a maximal ideal in $\mathbb {BC}$ as it properly contained in $I_1$ and $I_2$.\\
 Thus there does not exist any one-to-one correspondance between non-zero multiplicative linear functionals and maximal ideals of bicomplex Banach algebra $\mathcal A$. 
  
\begin{remark}In a unital complex Banach algebra, a proper ideal does not contain an invertible element. But in a unital bicomplex Banach algebra, a proper ideal may contain an invertible element.
\end{remark}
For example, in the bicomplex Banach algebra $\mathbb {BC}$ with unity $1$, $I_1=e_1\mathbb {BC}$ is a proper ideal in $\mathbb {BC}$. Further, all the elements in $I_1$ are invertible elements with multiplicative identity $e_1$. For this, let $e_1Z=e_1(e_1z_1+e_2z_2)=e_1z_1\in I_1$. Since $z_1\in \mathbb C(i)$, there exist some $w\in \mathbb C(i)$ such that $z_1w=wz_1=1$. Thus $e_1w\in I_1$ with $e_1w.e_1z_1=e_1z_1.e_1w=e_1$.

\end{section}		 

   \begin{section} {Composition Operators on $H^2(\mathbb{U}_{\mathbb{B}\mathbb{C}})$} 
 In this section we study the composition operator on bicomplex Hardy space $H^2(\mathbb{U}_{\mathbb{B}\mathbb{C}})$ and discuss the growth properties of the functions in $H^2(\mathbb{U}_{\mathbb{B}\mathbb{C}})$.\\
Let $\mathbb{U}_{\mathbb{B}\mathbb{C}}=\left\{Z=z+jw=e_1z_1+e_2z_2\;|\;(z_1,z_2)\in \mathbb{U}^2\right\}$
 be the unit disc in $\mathbb{B}\mathbb{C}$,
where $\mathbb{U}$ is the unit disk in the complex plane $\mathbb C(i)$ and $\mathbb{U}^2=\mathbb{U}\times\mathbb{U}$. The bicomplex Hardy space $H^2(\mathbb{U}_{\mathbb{B}\mathbb{C}})$ is defined to be the space of all bicomplex holomorphic functions on the unit disc $\mathbb{U}_{\mathbb{B}\mathbb{C}}$ whose sequence of power series coefficients is square-summable (see, \cite [P. 92] {YY}).
\begin{definition} Let $f:\mathbb{U}_{\mathbb{B}\mathbb{C}}\rightarrow \mathbb{B}\mathbb{C}$ be a bicomplex holomorphic function and $\Phi:\mathbb{U}_{\mathbb{B}\mathbb{C}}\rightarrow \mathbb{U}_{\mathbb{B}\mathbb{C}}$ be a bicomplex holomorphic self map of $\mathbb{U}_{\mathbb{B}\mathbb{C}}$.  
 The composition operator $C_\Phi$ on $H^2(\mathbb{U}_{\mathbb{B}\mathbb{C}})$ induced by $\Phi$   is defined by $$ C_\Phi f=f\;o\;\Phi, $$ for every $f\in H^2(\mathbb{U}_{\mathbb{B}\mathbb{C}}).$
\end{definition}  
\begin{remark}\label{cmop} For any $Z=e_1z_1+e_2z_2\in \mathbb{U}_{\mathbb{B}\mathbb{C}} $, the bicomplex holomorphic function $f$ on $\mathbb{U}_{\mathbb{B}\mathbb{C}}$ can be written as $f(Z)=e_1f_1(z_1)+e_2f_2(z_2),$ where $f_1$ and $f_2$ are complex-valued holomorphic functions on the unit disc $\mathbb{U}$. Similarly we can write $\Phi(Z)=e_1\Phi_1(z_1)+e_2\Phi_2(z_2)$, where $\Phi_1$ and $\Phi_2$ are holomorphic self maps of the unit disc $\mathbb{U}$. Thus the composition operator $C_\Phi$ on $H^2(\mathbb{U}_{\mathbb{B}\mathbb{C}})$ can also be defined as $ C_\Phi= e_1C_{\Phi_1}+e_2 C_{\Phi_2},$ where $C_{\Phi_l}$ for $l=1,2$ is a composition operator on $H^2(\mathbb{U}).$ For details on composition operators, one can refer to \cite{EE}, \cite{SS} and the references therein. 
\end{remark} 
\begin{proposition}\label{GE} Let $H^2(\mathbb U_{\mathbb B \mathbb C})$ be the bicomplex Hardy space. Then for each $f \in H^2(\mathbb U_{\mathbb B \mathbb C})$,$$ |f(Z)|_k \leq' \frac{\|f\|_{\mathbb{D},{H^2(\mathbb{U}_{\mathbb{B}\mathbb{C}})}}}{\sqrt{1-|Z|^2_k}},\;\; \forall Z\in \mathbb U_{\mathbb B \mathbb C}.$$
\end{proposition}
\begin{proof}  For each $f \in H^2(\mathbb U_{\mathbb B \mathbb C})$, we can write $f(Z)=e_1f_1(z_1)+e_2f_2(z_2)$, where $f_l(z_l)\in H^2(\mathbb U)$ for $l=1,2$. Further, by using the growth estimate of functions in $H^2(\mathbb U)$ (see \cite [P. 10] {SS}), each $f_l\in H^2(\mathbb U)$, $l=1,2$ satisfies $|f_l(z_l)|\leq \left(\frac{\|f_l\|_{H^2(\mathbb{U})}}{\sqrt{1-|z_l|^2}}\right)$.   
  Thus for each $Z\in \mathbb{U}_{\mathbb{B}\mathbb{C}}$, 	
\begin{align*}   |f(Z)|_k&=e_1|f_1(z_1)|+e_2|f_2(z_2)|\\
&\leq' e_1\left(\frac{\|f_1\|_{H^2(\mathbb{U})}}{\sqrt{1-|z_1|^2}}\right)+e_2\left(\frac{\|f_2\|_{H^2(\mathbb{U})}}{\sqrt{1-|z_2|^2}}\right)\\
&=\frac{\|f\|_{\mathbb{D},{H^2(\mathbb{U}_{\mathbb{B}\mathbb{C}})}}}{\sqrt{1-|Z|^2_k}}\;.
\end{align*}
Hence functions in bicomplex Hardy space also satisfy the growth condition. 
\end{proof}
\begin{corollary} Every $\mathbb D$-norm convergent squence in $H^2(\mathbb U_{\mathbb B \mathbb C})$ converges uniformly on compact subsets of $\mathbb U_{\mathbb B \mathbb C}$.
\end{corollary}
\begin{proof} Let $\left\{f_n\right\}$ be a sequence in $H^2(\mathbb U_{\mathbb B \mathbb C})$ converging in $\mathbb D$-norm to a function $f\in H^2(\mathbb U_{\mathbb B \mathbb C})$. Choose $r\in \mathbb D^+$ such that $r=r_1e_1+r_2e_2$ with $0<r_1<1$ and $0<r_2<1$. Consider the set $\left\{Z\in \mathbb U_{\mathbb B \mathbb C}\;|\; |Z|_k\leq' r\right\}=e_1\left\{z_1\in \mathbb U\;|\;|z_1|\leq r_1\right\}+e_2\left\{z_2\in \mathbb U\;|\;|z_2|\leq r_2\right\}$. Clearly $\left\{Z\in \mathbb U_{\mathbb B \mathbb C}\;|\; |Z|_k\leq' r\right\}$ is a closed set in $\mathbb U_{\mathbb B \mathbb C}$. Thus for each fixed $n$, Proposition \ref{GE} yields 
 $$\underset{|Z|_k \leq'r}{\mathop{\sup}} |f_{n}(Z)-f(Z)|_k  \leq' \frac{\|f_{n}-f\|_{\mathbb{D},{H^2(\mathbb{U}_{\mathbb{B}\mathbb{C}})}}}{\sqrt{1-r^2}}.$$ This means that $f_{n}\rightarrow f$ uniformly on the closed disc $\left\{|Z|_k \leq'r\right\}$. Since $r$ is arbitrary, hence $f_{n}\rightarrow f$ uniformly on every compact subset of $\mathbb U_{\mathbb B \mathbb C}$.
\end{proof}
\begin{example}\label{auto} Let $\mathbb U_{\mathbb {BC}}$ be the unit disc in $\mathbb {BC}$. Define $\Psi:\mathbb U_{\mathbb {BC}}\rightarrow \mathbb U_{\mathbb {BC}}$ as 
$$\Psi(Z)=\lambda \frac{Z+W}{1+W^{\dagger_3}Z},$$ where $|\lambda|_k=1$ and $W\in\mathbb U_{\mathbb {BC}} $. Then $\Psi$ is a bicomplex automorphism of  $\mathbb U_{\mathbb {BC}}$.
Writting $Z=e_1z_1+e_2z_2$, $W=e_1w_1+e_2w_2$ and $\lambda=e_1\lambda_1+e_2\lambda_2$, we have
  $$\lambda \frac{Z+W}{1+W^{\dagger_3}Z}=e_1\lambda_1 \frac{z_1+w_1}{1+\overline{w}_1z_1}+e_2\lambda_2 \frac{z_2+w_2}{1+\overline{w}_2z_2}.$$
Thus we can write $\Psi(Z)=e_1\Psi_1(z_1)+e_2\Psi_2(z_2),$ where $\Psi_l$  for $l=1,2$ is an automorphism of the unit disc $\mathbb U$.
\end{example}
\begin{theorem}\label{cowen} Let $H^2(\mathbb U_{\mathbb {BC}})$ be the bicomplex Hardy space and $\Psi:\mathbb U_{\mathbb {BC}}\rightarrow \mathbb U_{\mathbb {BC}}$ be an automorphism defined as in Example \ref{auto}. Then 
\begin{align*} \left(\frac{1-|\Psi(0)|_k}{1+|\Psi(0)|_k}\right)^{\frac{1}{2}}\|f\|_{\mathbb D,H^2(\mathbb U_{\mathbb B \mathbb C})}&\leq^{'} \|C_\Psi(f)\|_{\mathbb D,H^2(\mathbb U_{\mathbb B \mathbb C})}\\
& \leq^{'}\left(\frac{1+|\Psi(0)|_k}{1-|\Psi(0)|_k}\right)^{\frac{1}{2}}\|f\|_{\mathbb D,H^2(\mathbb U_{\mathbb B \mathbb C})}.  
\end{align*}
\end{theorem} 
\begin{proof}  Since we can write $\Psi(Z)=e_1\Psi_1(z_1)+e_2\Psi_2(z_2),$ where $\Psi_l$, $l=1,2,$ is an automorphism of the unit disc $\mathbb U$. Also by using Remark \ref{cmop}, composition operator $C_\Psi$ induced from $\psi$ on $H^2(\mathbb U_{\mathbb {BC}})$ can be written as $C_\Psi=e_1C_{\Psi_1}+e_2C_{\Psi_2}$ where $C_{\Psi_l}$, $l=1,2,$ is a composition operator on $H^2(\mathbb U)$.\\     
 Thus by using \cite [Theorem  3.6]{EE}, we have, for $l=1,2,$
$$\left(\frac{1-|\Psi_l(0)|}{1+|\Psi_l(0)|}\right)^{\frac{1}{2}}\|f_l\|_{H^2(\mathbb U)}\leq \|C_{\Psi_l}(f_l)\|_{H^2(\mathbb U)} \leq \left(\frac{1+|\Psi_l(0)|}{1-|\Psi_l(0)|}\right)^{\frac{1}{2}}\|f_l\|_{H^2(\mathbb U)}.$$
Hence the inequalities in the proof of the theorem follows easily by using the $\mathbb D$-valued norm on $H^2(\mathbb U_{\mathbb {BC}})$. 

\end{proof}
\begin{corollary} Let $H^2(\mathbb U_{\mathbb {BC}})$ be the bicomplex Hardy space and $\Phi:\mathbb U_{\mathbb {BC}}\rightarrow \mathbb U_{\mathbb {BC}}$ be a bicomplex holomorphic self map. Then 
$$\left(\frac{1}{1-|\Phi(0)|_k^2}\right)^{\frac{1}{2}} \leq^{'} \|C_\Phi\|_{\mathbb D,H^2(\mathbb U_{\mathbb B \mathbb C})} \leq^{'}\left(\frac{1+|\Phi(0)|_k}{1-|\Phi(0)|_k}\right)^{\frac{1}{2}}.$$  
\end{corollary}
   
\end{section}   

 \begin{section} {Bicomplex Weighted Hardy Space and Reproducing Kernels} 
 In this section we study the bicomplex version of the weighted Hardy spaces and discuss some of their properties. We also describe reproducing kernels for the bicomplex weighted Hardy spaces. The result of this section are the extensions of some results from \cite [Section 2.1]{EE}. 
 \begin{definition}  Consider a sequence $\left\{\beta(n)\right\}_{n=0}^\infty$ of positive hyperbolic numbers such that $\beta(n)=\|Z^n\|_\mathbb D$ with $Z\in\mathbb U_{\mathbb B\mathbb C}$. A bicomplex weighted Hardy space $H^2_\beta(\mathbb{U}_{\mathbb{B}\mathbb{C}})$ is a bicomplex Hilbert space consisting of bicomplex-valued holomorphic functions on the unit disc $\mathbb U_{\mathbb B\mathbb C}$ such that the set of monomials $\left\{1,\;Z,\; Z^2,.\;.\;.\right\}$ form an orthogonal basis in $H^2_\beta(\mathbb{U}_{\mathbb{B}\mathbb{C}})$. Thus, $f(Z)=\sum_{n=0}^\infty{a_nZ^n}$, with $a_n\in \mathbb {BC}$, represents a holomorphic function in $H^2_\beta(\mathbb{U}_{\mathbb{B}\mathbb{C}})$ if and only if the series of hyperbolic numbers $\sum_{n=0}^\infty{|a_n|^2_k\beta(n)^2}$ is convergent.\\ Further, for each $Z=e_1z_1+e_2z_2$ in $U_{\mathbb B\mathbb C}$, the orthogonality of $\left\{Z^n\right\}$ implies the set of monomials $\left\{z^n_l\right\}$ for $l=1,2$ is also orthogonal (see, \cite [P. 84] {GG}). Setting $a_n=e_1a_{n1}+e_2a_{n2}$ and $\beta(n)=\|Z^n\|_\mathbb D=e_1\|z_1^n\|_1+e_2\|z^n_2\|_2=e_1\beta_1(n)+e_2\beta_2(n)$, we get:  
 $$|a_n|^2_k\beta(n)^2=e_1|a_{n1}|^2\beta_1(n)^2+e_2|a_{n2}|^2\beta_2(n)^2.$$
This means $\sum_{n=0}^\infty{|a_n|^2_k\beta(n)^2}\; \text{is convergent}$ if and only if both the complex series $$\sum_{n=0}^\infty{|a_{n1}|^2\beta_1(n)^2},\;\; \sum_{n=0}^\infty{|a_{n2}|^2\beta_2(n)^2}$$ are convergent. Thus, both the functions 
$$f_1(z_1)=\sum_{n=0}^\infty{ a_{n1}z^n_1}\;\text{and} \;f_2(z_2)=\sum_{n=0}^\infty{a_{n2}z^n_2}$$ belong to the weighted Hardy space of the unit disk $H^2_{\beta_1}(\mathbb{U})$ and $H^2_{\beta_2}(\mathbb{U})$ respectively.
Hence the bicomplex weighted Hardy space can be written as
$$H^2_\beta(\mathbb{U}_{\mathbb{B}\mathbb{C}})=e_1H^2_{\beta_1}(\mathbb{U})+e_2H^2_{\beta_2}(\mathbb{U})\;.$$
The $\mathbb{B}\mathbb{C}$-valued inner product on $H^2_\beta(\mathbb{U}_{\mathbb{B}\mathbb{C}})$ is given by 
\begin{align*} 
<\sum_{n=0}^{\infty}a_nZ^n,\sum_{n=0}^{\infty}b_nZ^n>_{H^2_\beta(\mathbb{U}_{\mathbb{B}\mathbb{C}})}&= \sum_{n=0}^\infty{a_n\beta(n)\;.\;(b_n\beta(n))}^{\dagger_3}\\  
&=\sum_{n=0}^\infty a_n(b_n)^{\dagger_3}\beta(n)^2. 
\end{align*} 
Further, with the orthogonality of the monomials, the $\mathbb{B}\mathbb{C}$-valued inner product on $H^2_\beta(\mathbb{U}_{\mathbb{B}\mathbb{C}})$ induces the $\mathbb D$-valued norm:   
\begin{align*}
\|\sum_{n=0}^{\infty}a_nZ^n\|^2_{\mathbb{D},{H^2_\beta(\mathbb{U}_{\mathbb{B}\mathbb{C}})}}&= <\sum_{n=0}^{\infty}a_nZ^n,\sum_{n=0}^{\infty}a_nZ^n>_{H^2_\beta(\mathbb{U}_{\mathbb{B}\mathbb{C}})}=\sum_{n=0}^\infty{|a_n|^2_k\beta(n)^2_k}\\
&=e_1\sum_{n=0}^\infty{|a_{n1}|^2\beta_1(n)^2}+e_2\sum_{n=0}^\infty{|a_{n2}|^2\beta_2(n)^2}\\ 
&=e_1\|\sum_{n=0}^\infty{ a_{n1}z^n_1} \|^2_{H^2_{\beta_1}(\mathbb{U})}+e_2\|\sum_{n=0}^\infty{a_{n2}z^n_2}\|^2_{H^2_{\beta_2}(\mathbb{U})}\;. 
 \end{align*}
\end{definition}
  \begin{remark}\label{krnl} $(i)$ The bicomplex weighted Hardy space $H^2_\beta(\mathbb{U}_{\mathbb{B}\mathbb{C}})$ is a bicomplex Hardy space $H^2(\mathbb{U_{\mathbb{BC}}})$ with weight $\beta(n)\equiv 1$.\\  
 $(ii)$ The bicomplex weighted Hardy space $H^{2}_{\beta}(\mathbb{U_{\mathbb{BC}}})$ is a bicomplex Bergman space $A^{2}(\mathbb{U_{\mathbb{BC}}})$ with weight  $\beta(n)=(n+1)^{-1/2}$. Thus, the  bicomplex Bergman space $A^2(\mathbb{U}_{\mathbb{B}\mathbb{C}})$ is defined to be the set of functions $f:\mathbb{U}_{\mathbb{B}\mathbb{C}}\rightarrow \mathbb{B}\mathbb{C}$ such that for any $Z\in \mathbb{U}_{\mathbb{B}\mathbb{C}}$ and  $a_n \in \mathbb B \mathbb C$,
 $$f(Z)=\sum_{n=0}^\infty{a_nZ^n} $$ and the series of hyperbolic numbers $\sum_{n=0}^\infty\frac{|a_n|^2_k}{n+1}\; \text{is convergent}$. This means that the bicomplex Bergman space can be written as $$A^2(\mathbb{U}_{\mathbb{B}\mathbb{C}})=e_1A^2(\mathbb{U})+e_2A^2(\mathbb{U})\;,$$ where $A^2(\mathbb{U})$ is the Bergman space of the unit disk $\mathbb U$.\\  
$(iii)$ The bicomplex weighted Hardy space $H^{2}_{\beta}(\mathbb{U_{\mathbb{BC}}})$ is a bicomplex Dirichlet space $D(\mathbb{U_{\mathbb{BC}}})$ with weight   $\beta(n)=(n+1)^{1/2}$. Thus, the  bicomplex Dirichlet space $D(\mathbb{U}_{\mathbb{B}\mathbb{C}})$ is defined to be the set of functions $f:\mathbb{U}_{\mathbb{B}\mathbb{C}}\rightarrow \mathbb{B}\mathbb{C}$ such that for any $Z\in \mathbb{U}_{\mathbb{B}\mathbb{C}}$ and $a_n \in \mathbb B \mathbb C$
 $$f(Z)=\sum_{n=0}^\infty{a_nZ^n} $$ and the series of hyperbolic numbers $\sum_{n=0}^\infty|a_n|^2_k(n+1)\; \text{is convergent}$. Hence we can write the bicomplex Dirichlet space as $$D(\mathbb{U}_{\mathbb{B}\mathbb{C}})=e_1D(\mathbb{U})+e_2D(\mathbb{U})\;,$$ where $D(\mathbb{U})$ is the Dirichlet space of the unit disk $\mathbb U$.  
 \end{remark}  
   \begin{proposition}
Let $T$ be the operator of multiplication by $Z$ on $H^{2}_{\beta}(\mathbb{U_{\mathbb{BC}}})$. Then\\
$(i)$ $T$ is $\mathbb{D}$-bounded if and only if $\sup{\beta(n+1)/\beta(n}) \in \mathbb{D}$.\\ 
$(ii)$ $T$ is $\mathbb{D}$-bounded below if and only if $\inf{\beta(n+1)/\beta(n}) \in \mathbb{D}^{+}$.  
\end{proposition}
\begin{proof} For each $Z=e_1z_1+e_2z_2 \in \mathbb{U_{\mathbb{BC}}}$, let $T = e_1 T_1 + e_2 T_2$. Then $T_l$ is the operator of multiplication by $z_l$ on $H^{2}_{\beta_l}(\mathbb{U}),\; l=1,2.$\\
 $(i)$ We know that $T$ is $\mathbb{D}$-bounded if and only if $T_l$ is bounded, $l=1,2.$ Using \cite [Proposition 2.7]{EE} we have that $T_l$ is bounded if and only if $\sup{\beta_l(n+1)/\beta_l(n)<\infty}$ for $l=1,2$. It follows that
\begin{align*}
T~\textmd{is}~ \mathbb{D}-\textmd{bounded}
&\Leftrightarrow e_1 \sup{\beta_1(n+1)/\beta_1(n)}+e_2 \sup{\beta_2(n+1)/\beta_2(n)} \in \mathbb{D} \\
&\Leftrightarrow \sup\left(e_1 \beta_1(n+1)/\beta_1(n)+ e_2 \beta_2(n+1)/\beta_2(n)\right) \in \mathbb{D}\\
&\Leftrightarrow \sup{\beta(n+1)/\beta(n}) \in \mathbb{D}.
\end{align*} 
 
 $(ii)$ Clearly $T$ is $\mathbb{D}$-bounded below if and only if $T_l$ is bounded below, $l=1,2.$ By using \cite [Proposition 2.7]{EE} we see  that $T_l$ is bounded below if and only if $\inf{\beta_l(n+1)/\beta_l(n)}$ is positive, $l=1,2$. Therefore, we obtain
\begin{align*}
T\; \text{is}~ \mathbb{D}-\textmd{bounded below}
&\Leftrightarrow e_1 \inf{\beta_1(n+1)/\beta_1(n)}+e_2 \inf{\beta_2(n+1)/\beta_2(n)} \in \mathbb{D}^{+} \\
&\Leftrightarrow \inf\left(e_1 \beta_1(n+1)/\beta_1(n)+ e_2 \beta_2(n+1)/\beta_2(n)\right) \in \mathbb{D}^{+}\\
&\Leftrightarrow \inf{\beta(n+1)/\beta(n)} \in \mathbb{D}^{+}. 
\end{align*}
 \end{proof}
 \begin{definition} \label{DGnFn} Let $H^{2}_{\beta}(\mathbb{U_{\mathbb{BC}}})$ be a bicomplex weighted Hardy space. Then for each $Z\in \mathbb{U_{\mathbb{BC}}}$, the generating function for $H^{2}_{\beta}(\mathbb{U_{\mathbb{BC}}})$ is given by 
$$T(Z)=\sum_{n=o}^{\infty}\frac{Z^n}{\beta(n)^2}.$$
Setting $Z=e_1z_1+e_2z_2$ and $\beta(n)=e_1\beta_1(n) +e_2\beta_2(n)$ one gets: $$\sum_{n=o}^{\infty}\frac{Z^n}{\beta(n)^2}= e_1\sum_{n=o}^{\infty}\frac{z_1^n}{\beta_1(n)^2}+e_2\sum_{n=o}^{\infty}\frac{z_2^n}{\beta_2(n)^2}.$$ Thus, we can write $T(Z)= e_1T_1(z_1)+e_2T_2(z_2),$ where $T_1$ and $T_2$ are the generating functions for the weighted Hardy spaces $H^{2}_{\beta_1}(\mathbb{U})$   and $H^{2}_{\beta_2}(\mathbb{U})$ respectively. 
\end{definition}
 We now show that the generating function for the bicomplex weighted Hardy space $H^{2}_{\beta}(\mathbb{U_{\mathbb{BC}}})$ is holomorphic on the unit disk $\mathbb{U_{\mathbb{BC}}}$. 
\begin{lemma}\label{GnFn} Let $T$ be the generating function for $H^{2}_{\beta}(\mathbb{U_{\mathbb{BC}}})$. Then $T$ is holomorphic on the unit disk $\mathbb U_{\mathbb B\mathbb C}$.
\end{lemma}
\begin{proof} By using Definition \ref{DGnFn}, we can write $T(Z)= e_1T_1(z_1)+e_2T_2(z_2),$ where $T_1$ and $T_2$ are the generating functions for $H^{2}_{\beta_1}(\mathbb{U})$ and  $H^{2}_{\beta_2}(\mathbb{U})$ respectively. Since generating functions for complex weighted Hardy spaces are holomorphic on the open unit disk $\mathbb U$ (see \cite [Lemma 2.9]{EE}). Thus $T_1$ and $T_2$ are holomorphic on $\mathbb U$ and for each $z_1, z_2\in \mathbb U$, the respective Taylor series $\sum_{n=o}^{\infty}\frac{z_1^n}{\beta_1(n)^2}$ and $\sum_{n=o}^{\infty}\frac{z_2^n}{\beta_2(n)^2}$ of $T_1$ and $T_2$ are convergent. Hence for $Z=e_1z_1+e_2z_2 \in \mathbb{U_{\mathbb{BC}}}$, the Taylor series $\sum_{n=o}^{\infty}\frac{Z^n}{\beta(n)^2}=e_1\sum_{n=o}^{\infty}\frac{z_1^n}{\beta_1(n)^2}+e_2\sum_{n=o}^{\infty}\frac{z_2^n}{\beta_2(n)^2}$ is also convergent which implies $T$ is holomorphic on the unit disk $\mathbb U_{\mathbb B\mathbb C}$. 
\end{proof}
 A reproducing kernel Hilbert space is a Hilbert space associated with a kernel that reproduces every function in the space. Like classical weighted Hardy spaces, the bicomplex weighted Hardy space $H^{2}_{\beta}(\mathbb{U_{\mathbb{BC}}})$ also has an additional property of being a reproducing kernel Hilbert space. This means that for each $W\in \mathbb{U_{\mathbb{BC}}}$, the evaluation functional $E_W:H^{2}_{\beta}(\mathbb{U_{\mathbb{BC}}})\rightarrow \mathbb{BC}$ at $W$ defined by $E_W(f)=f(W)$ is bounded. Thus by using the Riesz representation theorem for bicomplex Hilbert spaces (see, \cite [P. 82] {GG}, \cite [P. 74] {YY}), there is a special function $K_W\in H^{2}_{\beta}(\mathbb{U_{\mathbb{BC}}})$ such that $f(W)=<f,K_W>_{H^2_\beta(\mathbb{U}_{\mathbb{B}\mathbb{C}})}$. In the next theorem we will see that the generating function $T$ generates the reproducing kernel for the  bicomplex weighted Hardy space $H^{2}_{\beta}(\mathbb{U_{\mathbb{BC}}})$.          
 
  \begin{theorem} \label{REPK} For each $W$ in $\mathbb{U_{\mathbb{BC}}}$, the evaluation of functions in $H^{2}_{\beta}(\mathbb{U_{\mathbb{BC}}})$ at $W$ is a $\mathbb D$-bounded bicomplex linear functional such that for every function $f$ in $H^{2}_{\beta}(\mathbb{U_{\mathbb{BC}}})$, $f(W)=<f,K_W>_{H^2_\beta(\mathbb{U}_{\mathbb{B}\mathbb{C}})}$ where $K_W(Z)=T(W^{\dagger_3}Z)$, $T$ is the generating function for $H^{2}_{\beta}(\mathbb{U_{\mathbb{BC}}})$.   
\end{theorem}
 \begin{proof} For each $W \in \mathbb U_{\mathbb {BC}}$,\begin{align*} \|K_W\|^2_{\mathbb{D},{H^2_\beta(\mathbb{U}_{\mathbb{B}\mathbb{C}})}}&= \|\sum_{n=0}^{\infty}\frac{\left(W^{\dagger_3}\right)^n}{\beta(n)^2}Z^n\|^2_{\mathbb{D},{H^2_\beta(\mathbb{U}_{\mathbb{B}\mathbb{C}})}}\\
&=\sum_{n=0}^{\infty}\left|\frac{\left(W^{\dagger_3}\right)^n}{\beta(n)^2}\right|_k^2\beta(n)^2=\sum_{n=0}^{\infty}\frac{ |W|_k^{2n}}{\beta(n)^2}=T(|W|^2_k).
 \end{align*} Now by Lemma \ref{GnFn}, the generating function $T$ for $H^{2}_{\beta}(\mathbb{U_{\mathbb{BC}}})$ is holomorphic on $\mathbb{U_{\mathbb{BC}}}$. Thus the series $\sum_{n=0}^{\infty}\frac{ |W|_k^{2n}}{\beta(n)^2}$ is convergent which implies $K_W \in H^{2}_{\beta}(\mathbb{U_{\mathbb{BC}}})$. 
 Further, if $f(Z)=\sum_{n=0}^{\infty} a_nZ^n \in H^{2}_{\beta}(\mathbb{U_{\mathbb{BC}}})$, then     
 \begin{align*} <f,K_W>_{H^2_\beta(\mathbb{U}_{\mathbb{B}\mathbb{C}})}&=\sum_{n=0}^\infty{a_n\beta(n)\left(\frac{\left(W^{\dagger_3}\right)^n}{\beta(n)^2}\beta(n)\right)}^{\dagger_3}\\
&= \sum_{n=0}^\infty{\frac{a_n W^n}{\beta(n)^2}\beta(n)^2}\\
&= \sum_{n=0}^\infty{a_n W^n}=f(W).
 \end{align*} 
  Hence evaluation of function $f$ in $H^{2}_{\beta}(\mathbb{U_{\mathbb{BC}}})$ is a bounded linear functional such that $f(W)=<f,K_W>_{H^2_\beta(\mathbb{U}_{\mathbb{B}\mathbb{C}})}$.
\end{proof}
 \begin{remark}\label{cor} From Theorem \ref{REPK}, for a bicomplex weighted Hardy space $H^{2}_{\beta}(\mathbb{U_{\mathbb{BC}}})$, the evaluation functional $E_W$ at $W$ is a bounded linear functional. Hence $H^{2}_{\beta}(\mathbb{U_{\mathbb{BC}}})$ is a bicomplex reproducing kernel Hilbert space with reproducing kernel $K_W(Z)=T(W^{\dagger_3}Z)$, where $T$ is the generating function for $H^{2}_{\beta}(\mathbb{U_{\mathbb{BC}}})$. Further, by using Definition \ref{DGnFn}, for each $W=e_1w_1+e_2w_2\in \mathbb{U_{\mathbb{BC}}}$,\\ we can write $T(W^{\dagger_3}Z)=e_1T_1(\overline{w}_1z_1)+e_2T_2(\overline{w}_2z_2)$. Thus reproducing kernel for $H^{2}_{\beta}(\mathbb{U_{\mathbb{BC}}})$ can also be defined as 
 $$ K_W(Z)=e_1K_{w_1}(z_1)+e_2K_{w_2}(z_2),$$
 where $K_{w_1}(z_1)$ and $K_{w_2}(z_2)$ are the reproducing kernels of the classical weighted Hardy spaces $H^2_{\beta_1}(\mathbb{U})$ and $H^2_{\beta_2}(\mathbb{U})$ respectively.   
 \end{remark}
 \begin{remark} Since a bicomplex weighted Hardy space $H^{2}_{\beta}(\mathbb{U_{\mathbb{BC}}})$ is a bicomplex Hardy space $H^2(\mathbb U_{\mathbb B \mathbb C})$ with weight $\beta(n)\equiv 1$. Hence $H^2(\mathbb U_{\mathbb B \mathbb C})$ is a bicomplex reproducing kernel Hilbert space. For each $W\in \mathbb U_\mathbb {BC}$, the reproducing kernel for $H^2(\mathbb U_{\mathbb B \mathbb C})$ is given by
 $$ K_W(Z)=T(W^{\dagger_3}Z)=\sum_{n=o}^{\infty}(W^{\dagger_3})^nZ^n=\frac{1}{1-W^{\dagger_3}Z}.$$ Further, writting $W=e_1w_1+e_2w_2$ and $Z=e_1z_1+e_2z_2$, we have
 
 \begin{align*} K_W(Z)&=\frac{1}{1-W^{\dagger_3}Z}=e_1\frac{1}{1-\overline{w}_1z_1}+e_2\frac{1}{1-\overline{w}_2z_2}\\
 &= e_1K_{w_1}(z_1)+e_2K_{w_2}(z_2),
 \end{align*}
where $K_{w_1}(z_1)$ and $K_{w_2}(z_2)$ are the reproducing kernels for the classical Hardy space $H^2(\mathbb U)\;.$ 
 \end{remark}
 \begin{corollary} The bicomplex Bergman space $A^2(\mathbb U_{\mathbb B \mathbb C})$ is a bicomplex reproducing kernel Hilbert space.
  \end{corollary}
  \begin{proof}  Since bicomplex Bergman space $A^2(\mathbb U_{\mathbb B \mathbb C})$ is a bicomplex weighted Hardy space $H^{2}_{\beta}(\mathbb{U_{\mathbb{BC}}})$ with weight $\beta(n)=(n+1)^{-1/2}$. Thus by using Remark \ref{cor}, $A^2(\mathbb U_{\mathbb B \mathbb C})$ is a bicomplex reproducing kernel Hilbert space. For each $W\in \mathbb U_\mathbb {BC}$, the reproducing kernel for $A^2(\mathbb U_{\mathbb B \mathbb C})$ is given by $$K_W(Z)=T(W^{\dagger_3}Z)=\sum_{n=o}^{\infty}(n+1)(W^{\dagger_3})^nZ^n=\frac{1}{(1-W^{\dagger_3}Z)^2}\;.$$
Writting $W=e_1w_1+e_2w_2$ and $Z=e_1z_1+e_2z_2$, the reproducing kernel for the bicomplex Bergman space $A^2(\mathbb U_{\mathbb B \mathbb C})$ can also be expressed as
   \begin{align*} K_W(Z)&=\frac{1}{(1-W^{\dagger_3}Z)^2}=e_1\frac{1}{(1-\overline{w}_1z_1)^2}+e_2\frac{1}{(1-\overline{w}_2z_2)^2}\\
 &= e_1K_{w_1}(z_1)+e_2K_{w_2}(z_2),
 \end{align*}
 where $K_{w_1}(z_1)$ and $K_{w_2}(z_2)$ are the reproducing kernels for the classical Bergman space $A^2(\mathbb U)\;.$ 
 \end{proof}
 \begin{remark} Since bicomplex Dirichlet space $D(\mathbb{U_{\mathbb{BC}}})$ is a bicomplex weighted Hardy space $H^{2}_{\beta}(\mathbb{U_{\mathbb{BC}}})$ with weight $\beta(n)=(n+1)^{1/2}$. Thus for each $W\in \mathbb U_\mathbb {BC}$, the reproducing kernel for $D(\mathbb{U_{\mathbb{BC}}})$ is given by $$K_W(Z)=T(W^{\dagger_3}Z)=\sum_{n=o}^{\infty}\frac{(W^{\dagger_3})^nZ^n}{n+1}=\frac{1}{W^{\dagger_3}Z}log\left(\frac{1}{1-W^{\dagger_3}Z}\right).$$
Further, we can also write \begin{align*} K_W(Z)&=\frac{1}{W^{\dagger_3}}log\left(\frac{1}{1-W^{\dagger_3}Z}\right)\\\  
 &=e_1\frac{1}{\overline{w}_1z_1}log\left(\frac{1}{1-\overline{w}_1z_1}\right)+e_2\frac{1}{\overline{w}_2z_2}log\left(\frac{1}{1-\overline{w}_2z_2}\right)\\
 &=e_1K_{w_1}(z_1)+e_2K_{w_2}(z_2),
 \end{align*}
 where $K_{w_1}(z_1)$ and $K_{w_2}(z_2)$ are the reproducing kernels for the classical Dirichlet space $D(\mathbb U).$
 \end{remark}    
   \end{section}

\bibliographystyle{amsplain}

\noindent Romesh Kumar, \textit{Department of Mathematics, University of Jammu, Jammu, J\&K - 180 006, India.}\\
E-mail :\textit{ romesh\_jammu@yahoo.com}\\

\noindent Kulbir Singh, \textit{Department of Mathematics, University of Jammu, Jammu,  J\&K - 180 006, India.}\\
E-mail :\textit{ singhkulbir03@yahoo.com}\\

\noindent Heera Saini, \textit{Department of Mathematics, University of Jammu, Jammu,  J\&K - 180 006, India.}\\
E-mail :\textit{ heerasainihs@gmail.com}\\

\noindent Sanjay Kumar, \textit{Department of Mathematics, Central University of Jammu, Jammu, J\&K - 180 006, India.}\\
E-mail :\textit{ sanjaykmath@gmail.com}

\end{document}